\documentclass[11pt]{amsart}
\usepackage{amsfonts}
\usepackage{mathrsfs}
\usepackage{amsmath}
\usepackage{amsmath, amsthm, amssymb}
\usepackage{color} 
\numberwithin{equation}{section} 
\renewcommand{\(}{\left( }
\renewcommand{\)}{\right) }

\renewcommand{\theequation}{\theequation. \arabic{equation}}
\numberwithin{equation}{section}
\newtheorem{thm}{Theorem}[section]

\newtheorem{prop}[thm]{Proposition}
\newtheorem{defn}[thm]{Definition}

\def\squarebox#1{\hbox to #1{\hfill\vbox to #1{\vfill}}}

\begin{document}
\title[On a $q$-beta integral with twelve parameters]
{Askey--Wilson polynomials and a double $q$-series transformation formula with twelve parameters}
\author{Zhi-Guo Liu}
\date{\today}
\address{School of Mathematical Sciences and S Shanghai  Key Laboratory of PMMP, East China Normal University, 500 Dongchuan Road,
Shanghai 200241, P. R. China} \email{zgliu@math.ecnu.edu.cn;
liuzg@hotmail.com}
\thanks{The  author was supported in part by
 the National Natural Science Foundation of China and Science and Technology Commission of Shanghai Municipality (Grant No. 13dz2260400)}
\thanks{ 2010 Mathematics Subject Classifications :  05A30,
33D15, 33D45, 11E25.}
\thanks{ Keywords:  Askey--Wilson integral, Askey--Wilson polynomials, Nassrallah--Rahman integral, $q$-beta integral.}
\begin{abstract}
The Askey--Wilson polynomials are the most general classical orthogonal
polynomials that are known and the Nassrallah--Rahman integral is a very
general extension of Euler's integral representation of the classical $_2F_1$ function.
Based on a $q$-series transformation formula and the Nassrallah--Rahman integral
we prove a $q$--beta integral which has twelve parameters, with several other results,
both classical and new, included as special cases. This $q$-beta integral also allows
us to  derive  a curious double $q$--series transformation formula, which includes one
formula of Al--Salam and Ismail as a special case.
\end{abstract}
\maketitle

\section{Introduction and preliminaries}
\noindent
Askey and Wilson \cite{AskWil1985} use the Askey--Wilson integral to find a set of orthogonal polynomials,
which are now known as the Askey--Wilson polynomials. These polynomials
are the most general classical orthogonal polynomials that are known. It is well-known
that \cite{Ismail2009} the Askey--Wilson polynomials have five parameters, one of which is special, usually denoted by $q$, called a base. Ismail and Wilson \cite{IsmailWilson1982}  found  a generating function for the Askey--Wilson polynomials. Szablowski \cite{Szablowski2011}  discussed some equivalent forms  for the Askey--Wilson polynomials. Liu \cite{Liu2013RamJ}  found  a new generating function of the Askey--Wilson polynomials, which  was used to give a simple proof of the orthogonality for the Askey--Wilson polynomials.

It is generally known that it is a diffcult task to obtain double $q$-series
	transformation formulae. Al--Salem and Ismail \cite{Al--SalamIsmail1988} use $q$-Hermite polynomials to give an alternative
evaluation of the Nassrallah--Rahman integral, and as a  byproduct, they derive
a double $q$-series transformation formula with six parameters. In this paper we will use
the orthogonality relation for the Askey--Wilson polynomials to give a completely
new proof of the Nassrallah--Rahman integral and then use this integral to prove a $q$-beta integral formula with $12$ parameters which  includes the
Askey--Wilson integral as a special case.

This $q$-beta integral formula allows us to
prove a double $q$-series transformation formula which contains the Al--Salem and Ismail
double $q$-formula and some other double $q$-transformation formulas as
special cases. We believe that our method is a new way of deriving $q$-formulas.

We now introduce some notations. For $0<q<1$,  the $q$-shifted factorials  are defined as
\begin{equation*}
(a; q)_0=1,\quad (a; q)_n=\prod_{k=0}^{n-1}(1-aq^k), \quad (a;
q)_\infty=\prod_{k=0}^\infty (1-aq^k).
\end{equation*}
If $n$ is an integer or $\infty$ and $m$ is a positive integer, the multiple
$q$-shifted factorial are defined as 
\begin{equation*}
(a_1, a_2,...,a_m;q)_n=(a_1;q)_n(a_2;q)_n ... (a_m;q)_n.
\end{equation*}
As usual, the basic hypergeometric series or the $q$-hypergeometric series $_{r+1}\phi_r$ is
defined by
\begin{equation*}
_{r+1}\phi_r \left({{a_1, a_2, ..., a_{r+1}} \atop {b_1, b_2, ...,
		b_r}} ;  q, z  \right) =\sum_{n=0}^\infty \frac{(a_1, a_2, ...,
	a_{r+1};q)_n z^n} {(q,  b_1, b_2, ..., b_r ;q)_n}.
\end{equation*}
The above $q$-hypergeometric series 
is called  well-poised if the parameters satisfy the relations $qa_1=a_2b_1=a_3b_2=\ldots=a_{r+1}b_r;$
very-well-poised if, in addition, 
$a_2=q\sqrt{a_1}$ and $a_3=-q\sqrt{a_1}.$ 

For notational simplicity, we sometimes use the more compact notation
\[
 _{r+1}W_r(a_1; a_4, a_5, \ldots, a_{r+1}; q, z)
\]
 to denote the very-well-poised series
\begin{equation*}
_{r+1}\phi_r \left({{a_1, q\sqrt{a_1}, -q\sqrt{a_1},..., a_{r+1}} \atop {\sqrt{a_1}, -\sqrt{a_1}, qa_1/a_4...,
		qa_1/a_{r+1}}} ;  q, z  \right) .
\end{equation*}

For $x=\cos \theta$, the Askey--Wilson polynomials $p_n(\cos \theta; a, b,c, d |q)$ are defined as \cite{AskWil1985}, 
\cite[p. 188]{GasperRahman2004}

\begin{equation*}
(ab, ac, ad; q)_n a^{-n} {_4\phi_3}\left({{q^{-n}, abcdq^{n-1}, ae^{-i\theta}, ae^{i\theta}}\atop{ab, ac, ad}}; q, q\right).
\end{equation*}
\begin{defn}\label{quadraticq-shift:defn}
For $x=\cos \theta$, we define $h(x; a|q)$ and 
$h(x; a_1, \ldots, a_m|q)$ as 	
\begin{align*}h(x; a|q)=(ae^{i\theta}, ae^{-i\theta}; q)_\infty
=\prod_{k=0}^\infty (1-2q^k ax+q^{2k} a^2),\\
h(x; a_1, \ldots, a_m|q)
=h(x; a_1|q)h(x; a_2|q)\cdots h(x; a_m|q).
\end{align*}
\end{defn}
For notational simplicity, we will use $W_{(a, b,c, d)}(\cos \theta|q)$ to denote 
\[
\frac{h(\cos 2\theta; 1|q)}{h(\cos \theta; a, b, c, d|q)}
\]
or
\[
\frac{(e^{2i\theta}, e^{-2i\theta}; q)_\infty}
{(ae^{i\theta}, ae^{-i\theta}, be^{i\theta}, be^{-i\theta}, ce^{i\theta}, ce^{-i\theta}, de^{i\theta}, de^{-i\theta}; q)_\infty}.
\]

\begin{prop}\label{AskWilint} For $\max\{|a|, |b|, |c|,|d|\}<1$, the Askey--Wilson  $q$-beta integral formula  states that
\[
\int_{0}^{\pi} W_{(a, b, c, d)}(\cos\theta|q) d\theta
=K(a, b, c, d|q),
\]
 where here, and throughout the paper, 
$K(a, b, c, d|q)$ is given by
\[
K(a, b, c, d|q)=\frac{2\pi (abcd; q)_\infty}{(q, ab, ac, ad, bc, bd, cd; q)_\infty}.
\]
\end{prop}
There are several different ways of proving the Askey--Wilson $q$-beta integral formula,
see, for example \cite{Askey1983, AskWil1985, IsmailStanton1987, KalninsMiller1989, Liu1997,  LiuRamJ2009, LiuZeng2015, Liu2016Symmetry,  Rahman1984}.

The orthogonality relation for the Askey--Wilson polynomials is stated in the following
theorem (see, for example \cite[p.384]{Ismail2009} and \cite[p.416]{KoekoekLesky2010}).
\begin{thm}\label{AskWilOrth}
	Let $p_n (\cos \theta; a, b, c, d|q)$ be the Askey--Wilson polynomials and $\delta_{m, n}$ the Kronecker delta. Then, for $\max\{|a|, |b|, |c|, |d|\}<1$,  we have 
\begin{align*}
&\int_{0}^{\pi} W_{(a, b, c, d)}(\cos \theta|q)p_n(\cos \theta; a, b, c,d |q) p_m(\cos\theta; a, b, c, d|q) d\theta \\
&=K(a, b,c, d|q)\frac{(1-abcd/q) (q, ab, ac, ad, bc, bd, cd; q)_n}{
	(1-abcdq^{2n-1}) (abcd/q; q)_n} \delta_{m,n}.
\end{align*}
\end{thm}
When $m=n=0$ the orthogonality relation for the Askey--Wilson polynomials
reduces to the Askey--Wilson integral.

The Rogers $_6\phi_5$ summation formula (see, for example \cite[p.36]{GasperRahman2004})  is stated in the following proposition.
\begin{prop}\label{RogersSum} For $|{q\alpha}/{bcd}|<1, $ we have the $q$-summation formula
\begin{equation*}
_6W_5 \left(\alpha; b, c, d; q, \frac{q\alpha} {bcd}\right)
=\frac{(q\alpha, \alpha q/bc, \alpha q/bd, \alpha q/cd; q)_\infty}{(\alpha q/b, \alpha q/c, \alpha q/d, \alpha q/bcd; q)_\infty}.
\end{equation*}
\end{prop}	
We  proved the following extension of the Rogers summation by using
$q$-exponential differential operator in \cite[Theorem~3]{Liu2011}. For an alternative proof, see \cite[Theorem~6.2]{Liu2013RamLegacy}.
\begin{thm}\label{rogersliuthm}
	For $\max\{|\alpha \beta abc/q^2|, |\alpha \gamma abc/q^2|\}<1$,  we have the identity
	\begin{align*}
	&\sum_{n=0}^\infty \frac{(1-\alpha q^{2n})(\alpha, q/a, q/b, q/c; q)_n}{(q, \alpha a, \alpha b, \alpha c; q)_n}
	\(\frac{\alpha abc}{q^2}\)^n \\
	&\qquad \qquad \qquad  \times  {_4 \phi_3} \left({{q^{-n}, \alpha q^n, \beta, \gamma}
		\atop{q/a, q/b,\alpha \beta \gamma ab/q}}; q, q\right)\\
	&=\frac{(\alpha, \alpha ac/q, \alpha bc/q, \alpha \beta ab/q, \alpha \gamma ab/q, \alpha \beta \gamma abc/q^2; q)_\infty}
	{(\alpha a, \alpha b, \alpha c, \alpha \beta abc/q^2, \alpha \gamma abc/q^2, \alpha \beta \gamma ab/q; q )_\infty}.
	\end{align*}
\end{thm}
When $\gamma=0$, this theorem reduces to the $q$-summation formula due to Ismail--
Rahman--Suslov \cite{IsmailRahmanSuslov1997}.

Upon replacing $c$ by $qt$ and then $(a, b)$ by $(q/ab, q/ac)$ in the equation above, we find that
	\begin{align*}
&\sum_{n=0}^\infty \frac{(1-\alpha q^{2n})(\alpha, ab, ac, t^{-1}; q)_n}{(1-\alpha)(q, q\alpha/ab, q\alpha/ac, \alpha tq; q)_n}
\(\frac{\alpha tq}{a^2bc}\)^n \\
&\qquad \qquad \qquad  \times  {_4 \phi_3} \left({{q^{-n}, \alpha q^n, \beta, \gamma}
	\atop{ab, ac, q\alpha \beta \gamma/a^2bc}}; q, q\right)\\
&=\frac{(q\alpha, q\alpha t/ab, q\alpha t/ac, q\alpha \beta/a^2bc, q\alpha \gamma /a^2bc, q\alpha \beta \gamma t/a^2bc; q)_\infty}
{(q\alpha/ ab, q\alpha/ac, q\alpha t , q\alpha \beta t/a^2bc, q\alpha \gamma t/a^2bc, q\alpha \beta \gamma /a^2bc; q )_\infty}.
\end{align*}

The following proposition \cite[Theorem~4]{Liu2011} can be derived from the above equation by  replacing $(\alpha, \beta, \gamma)$ with $(abcd/q, ae^{i\theta}, ae^{-i\theta})$.

\begin{prop} \label{LiuAWG} For $|dt|<1,$ the Askey--Wilson polynomials satisfy
\begin{align*}
\sum_{n=0}^\infty \frac{(1-abcdq^{2n-1})(abcd/q, t^{-1})_n (dt)^n}{(1-abcd/q)(q, ad, bd, cd, abcdt; q)_n}
p_n(\cos \theta; a, b, c, d|q)\\
=\frac{(abcd, adt, bdt, cdt,  de^{i\theta}, de^{-i\theta}; q)_\infty}{(abcdt, ad, bd, cd, dte^{i\theta}, dte^{-i\theta}; q)_\infty}.
\end{align*}		
\end{prop}
Professor Ismail kindly point out that this generating  function  follows  from
general  theory  of  connection  relations  and  from  the  connection  relation  in
the Askey--Wilson memoir.
	
This  generating function of the Askey--Wilson polynomials implies the
following proposition \cite[Corollary~4]{Liu2011}.
\begin{prop}\label{AskWilSymmetry}
The Askey--Wilson polynomials $p_n(\cos \theta; a, b, c, d|q)$ are symmetric
in $a, b, c$ and $d$.	
\end{prop}
Interchanging $a$ and $d$ in Proposition~\ref{LiuAWG} and noting  $p_n(\cos \theta; a, b, c, d|q)$ are symmetric
symmetric in $a$ and $d$, we deduce  the following proposition.

\begin{prop} \label{LiuAWGNew} For $|at|<1,$ the Askey--Wilson polynomials satisfy
	\begin{align*}
	\sum_{n=0}^\infty \frac{(1-abcdq^{2n-1})(abcd/q, t^{-1})_n (at)^n}{(1-abcd/q)(q, ab, ac, ad, abcdt; q)_n}
	p_n(\cos \theta; a, b, c, d|q)\\
	=\frac{(abcd, abt, act, adt,  ae^{i\theta}, ae^{-i\theta}; q)_\infty}{(abcdt, ab, ac, ad, ate^{i\theta}, ate^{-i\theta}; q)_\infty}.
	\end{align*}		
\end{prop}
This generating function  was used by us to give a new proof of the orthogonality for the Askey--Wilson polynomials \cite[pp. 199--200]{Liu2013RamJ}.

In Section~2 of this paper we will use Proposition~\ref{LiuAWG} and the orthogonality relation
for the Askey--Wilson polynomials to give a new proof of the Nassrallah--Rahman
integral. Section~3 is devoted to some inequalities for $q$-series. In Section~4, we will
prove the following $q$-beta integral which has twelve parameters with base $q$.

\begin{thm}\label{Liu12parametersint}
For $\alpha=a^2bcdr/q$ and $\max\{|a|, |b|, |c|,|d|, |r|, |s|, |t|, |h|, |z|\}<1,$	we have 
\begin{align*}
&\int_{0}^{\pi} \frac{h(\cos 2\theta; 1|q)}{h(\cos \theta; a, b, c, d, r|q)}
 {_4 \phi_3} \left({{ae^{i\theta}, a e^{-i\theta}, \beta, \delta}
	\atop{s, t, h}}; q, bcdrz\right)d\theta\\
&=\frac{2\pi(abcd, abcr, abdr, acdr; q)_\infty}{(q, ab, ac, ad, ar, bc, bd, br, cd, cr, dr, q\alpha; q)_\infty}\\
&\qquad \times 
\sum_{n=0}^\infty \frac{(1-\alpha q^{2n})(\alpha, ab, ac, ad, ar; q)_n}{(1-\alpha)(q, abcd, abcr, abdr, acdr; q)_n}(-bcdr)^n q^{n(n-1)/2}\\
&\qquad \qquad \qquad \qquad \qquad \times 
{_4 \phi_3} \left({{q^{-n}, \alpha q^{n}, \beta, \delta}
	\atop{s, t, h}}; q, qz\right).
\end{align*}
The theorem also holds when $|z|=1$, provided the series on the right-hand side of
the above equation converges absolutely for 
$|z|=1.$
\end{thm}
 Note that when $r=0$, the above integral immediately reduces to the
Askey--Wilson integral in Proposition~\ref{AskWilint}. In Section~5, we will show that Theorem~\ref{Liu12parametersint} is equivalent to
the following double $q$-series transformation formula.
\begin{thm}\label{Liu12parametersdoub}
	For $\alpha=a^2bcdr/q$ and $\max\{|a|, |b|, |c|,|d|, |r|, |s|, |t|, |h|, |z|\}<1,$	we have
\begin{align*}
&\sum_{n=0}^\infty \frac{(\beta, \delta, ab, ac, ad, ar; q)_n (bcdrz)^n}{(q, s, t, h, abcd, abcr; q)_n}
{_3 \phi_2} \left({{abq^{n}, acq^{n}, bc}
	\atop{abcdq^n, abcrq^n}}; q, dr\right)\\
&=\frac{(abdr, acr; q)_\infty}{(dr, q\alpha; q)_\infty}
\sum_{n=0}^\infty \frac{(1-\alpha q^{2n})(\alpha, ab, ac, ad, ar; q)_n}{(1-\alpha)(q, abcd, abcr, abdr, acdr; q)_n}(-bcdr)^n q^{n(n-1)/2}\\
&\qquad \qquad \qquad \qquad \qquad \times 
{_4 \phi_3} \left({{q^{-n}, \alpha q^{n}, \beta, \delta}
	\atop{s, t, h}}; q, qz\right).
\end{align*}	
The theorem also holds when $|z|=1$, provided the series on the right-hand side of
the above equation converges absolutely for 
$|z|=1.$
\end{thm}
Some interesting special cases of Theorem~\ref{Liu12parametersint} will be discussed in Section~6.

\section{A new proof of the Nassrallah--Rahman integral}
Nassrallah and Rahman \cite[pp.192--194]{NassrallahRahman1985} used the integral representation of the
sum of two non-terminating $_3\phi_2$ series and the Askey--Wilson integral to find the
following $q$-beta integral formula, which is a very general
extension of Euler's integral representation of the classical $_2F_1$ function. Ismail and
Zhang \cite{IsmailZhang2005} evaluate the Nassrallah--Rahman integral in a clever way by using the
method of analytic continuation.
In this section we will use Proposition~\ref{LiuAWG} and the orthogonality relation for the
Askey--Wilson polynomials to give a new proof of the Nassrallah--Rahman integral.
Our proof is somewhat simpler than that of Nassrallah and Rahman.
\begin{thm}\label{NassRahInt} For $\{|a|, |b|, |c|,|u|, |v|\}<1,$ we have the integral formula	
\begin{align*}
&\int_{0}^{\pi} \frac{h(\cos 2\theta; 1|q)h(\cos \theta; d|q)}
{h(\cos \theta; a, b, c, u, v|q)} d\theta \\
&=\frac{2\pi(abcu, abcv, ad, bd, cd; q)_\infty}{(q, abcd, ab, ac, au, av,  bc, bu, bv,  cu,  cv; q)_\infty}\\
&\quad \times {_8W_7}\left(abcd/q; ab, ac, bc, d/u, d/v; q, uv\right).
\end{align*}	
\end{thm}
\begin{proof}
For the sake of brevity,  we for the time being  use $p_n(\cos \theta|q)$ to denote $p_n(\cos \theta; a, b, c, d|q)$ and  $\Delta_n(t)$ to denote	
\[
\frac{(1-abcdq^{2n-1})(abcd/q, t^{-1})_n (dt)^n}{(1-abcd/q)(q, ad, bd, cd, abcdt; q)_n}.
\]
 Keeping  this in mind, we can rewrite the equation in Proposition~\ref{LiuAWG} in the form
\begin{align*}
\sum_{n=0}^\infty \Delta_n(t)
p_n(\cos \theta|q)
=\frac{(abcd, adt, bdt, cdt; q)_\infty h(\cos \theta; d)}{(abcdt, ad, bd, cd; q)_\infty h(\cos \theta; dt)}.
\end{align*}
If $t$ is replaced by $s$ in the above equation and making the variable change $n \to m$,
then, we obtain
\begin{align*}
\sum_{m=0}^\infty \Delta_m(s)
p_m(\cos \theta|q)
=\frac{(abcd, ads, bds, cds; q)_\infty h(\cos \theta; d)}{(abcdt, ad, bd, cd; q)_\infty h(\cos \theta; ds)}.
\end{align*}
Upon multiplying the two equations above together, we immediately conclude that
\begin{align*}
&\sum_{m=0}^\infty \sum_{n=0}^\infty\Delta_m(s)\Delta_n(t)
p_m(\cos \theta|q)p_n(\cos \theta|q)\\
&=\frac{(abcd, abcd,  ads, adt,  bds, bdt, cds, cdt; q)_\infty h(\cos \theta; d, d|q)}{(abcds,abcdt,  ad, ad,  bd, bd, cd, cd; q)_\infty h(\cos \theta; ds, dt)}.
\end{align*}
Multiplying the above equation by $W_{(a, b, c, d)}(\cos \theta|q)$ and then integrating with respect to $\theta$ over $[0, \pi]$, we deduce that
\begin{align*}
&\sum_{m=0}^\infty \sum_{n=0}^\infty\Delta_m(s)\Delta_n(t)
\int_{0}^{\pi} p_m(\cos \theta|q)p_n(\cos \theta|q)W_{(a, b, c, d)}(\cos \theta|q) d\theta\\
&=\frac{(abcd, abcd,  ads, adt,  bds, bdt, cds, cdt; q)_\infty }{(abcds,abcdt,  ad, ad,  bd, bd, cd, cd; q)_\infty}
\int_{0}^{\pi}
\frac{h(\cos 2\theta; 1) h(\cos \theta; d|q)}
{h(\cos \theta; a, b, c, ds, dt|q)} d\theta.
\end{align*}
Using the orthogonality relation for the Askey--Wilson polynomials to simplify the
left-hand side of the above equation, we find that
\begin{align*}
&K(a, b,c, d|q)\sum_{n=0}^\infty 
\frac{(1-abcd/q)(q, ab, ac, ad, bc, bd, cd; q)_n }{(1-abcdq^{2n-1})(abcd/q; q)_n}\Delta_n(s)\Delta_n(t)\\
&=K(a, b, c, d|q)~ {_8W_7} \left(abcd/q; ab, ac, bc, s^{-1}, t^{-1}; q, std^2\right).
\end{align*}
It follows that 
\begin{align*}
&\int_{0}^{\pi}
\frac{h(\cos 2\theta; 1) h(\cos \theta; d|q)}
{h(\cos \theta; a, b, c, ds, dt|q)} d\theta\\
&=\frac{2\pi (abcds, abcdt, ad, bd, cd; q)_\infty}{(q, abcd, ab, ac, bc, ads, bds, cds, adt, bdt, cdt; q)_\infty}\\
&\qquad \times {_8W_7} \left(abcd/q; ab, ac, bc, s^{-1}, t^{-1}; q, std^2\right).
\end{align*}
This completes the proof of Theorem~\ref{NassRahInt} after replacing $ds$ by $u$ and $dt$ by $v$.
\end{proof}	
When $d=abcuv,$  using Proposition~\ref{RogersSum},  the $_8W_7(\cdot)$ series in the Nassrallah--Rahman integral reduces to 
\[
_6W_5 \left(abcd/q; ab, ac, bc; q, uv\right)
=\frac{(abcd, bcuv, acuv, abuv; q)_\infty}
{(ad, bd, cd, uv; q)_\infty}.
\]
Hence Theorem~\ref{NassRahInt} reduces to the following $q$-integral formula, which was first derived by  Rahman \cite[Eq.(2.1)]{Rahman1986}.
\begin{thm}\label{RahInt} For $\{|a|, |b|, |c|,|u|, |v|\}<1,$ we have the integral formula	
	\begin{align*}
	&\int_{0}^{\pi} \frac{h(\cos 2\theta; 1|q)h(\cos \theta; abcuv|q)}
	{h(\cos \theta; a, b, c, u, v|q)} d\theta \\
	&=\frac{2\pi(abcu, abcv, abuv, acuv, bcuv; q)_\infty}{(q, ab, ac, au, av, bc, bu, bv, cu, cv, uv; q)_\infty}.
	\end{align*}	
\end{thm}
Upon setting $d=0$ in the Nassrallah--Rahman integral in Theorem~\ref{NassRahInt}, we immediately
arrive at the Ismail--Stanton--Viennot integral \cite[Theorem~3.5]{IsmailStanton1987}.
\begin{thm}\label{IsmailStanton} For $\{|a|, |b|, |c|,|u|, |v|\}<1,$ we have the integral formula	
	\begin{align*}
	&\int_{0}^{\pi} \frac{h(\cos 2\theta; 1|q)}
	{h(\cos \theta; a, b, c, u, v|q)} d\theta \\
	&=\frac{2\pi(abcu, abcv; q)_\infty}{(q, ab, ac, au, av, bc, bu, bv, cu, cv; q)_\infty}
	{_3\phi_2\left({{ab, ac, bc}\atop{abcu, abcv}}; q, uv\right)}.
	\end{align*}	
\end{thm}
\section{some inequalities for $q$-series}
The following proposition is proved by us in \cite[ Proposition~3.1]{Liu2016JMAA}. For completeness,
we will repeat the proof of this proposition.
\begin{prop}\label{liuequality} If $0<q<1$ and $k$ is a nonnegative integer or $\infty$, $a$ and $b$ are two nonnegative numbers
	such that $0 \le b \le 1, $ then, we have
	\[
	(-ab; q)_k \le (-a; q)_\infty.
	\]
	If we further assume that $0\le a \le 1$,  then, we have the inequality
	\[
	(ab; q)_k \ge (a; q)_\infty.
	\]
	
\end{prop}
\begin{proof} Keeping  that $0<q<1$ in mind, we find that for  $j\in \{ 0, 1,\ldots, k-1 \}$,
	\[
	1+abq^j \le 1+aq^j.
	\]
	Upon multiplying these $k$ inequalities together, we immediately deduce that
	\[
	(-ab; q)_k \le (-a; q)_k.
	\]
	Since $(-aq^k; q)_\infty\ge 1,$ we multiply $(-aq^k; q)_\infty$ to the right-hand side of the above inequality
	to arrive at the first inequality in the proposition. In the same way we can prove the second inequality.
	This completes the proof of Proposition~\ref{liuequality}.
\end{proof}
\begin{prop}\label{convergenceseries} If $\max\{|b_1|, |b_2|, \ldots, |b_r|\}<1, |x|\le \lambda<1$ and $n$ is a nonnegative integer,
	then, we have
	\[
	q^{n\choose 2}\left|{_{r+1}\phi_r}\left({{q^{-n},  a_1q^n, a_2,  \ldots, a_r }\atop{b_1, \ldots b_r }}; q, qx\right)\right|
	\le \frac{(-|a_1\lambda|, -q, -|a_2|, \ldots, -|a_r|; q)_\infty}{(|\lambda|, |b_1|, \ldots, |b_r|; q)_\infty}.
	\]
\end{prop}
 \begin{proof}
 By a simple evaluation, we easily find that 
 (see, for example \cite[Eq.(1.2.37)]{GasperRahman2004})
 \[
 (q^{-n}; q)_k=(-1)^k q^{-nk+{k\choose 2}}
 (1-q^n)(1-q^{n-1})\cdots (1-q^{n-k+1}).
 \]
 It follows that
 \begin{align*}
 &	q^{n\choose 2}{_{r+1}\phi_r}\left({{q^{-n},  a_1q^n, a_2,  \ldots, a_r }\atop{b_1, \ldots b_r }}; q, qx\right)\\
 &=\sum_{k=0}^n (-1)^k q^{{n-k}\choose 2}
 \frac{(q^{n-k+1}, a_1q^n, a_2, \ldots, a_r; q)_k}{(q, b_1,\ldots, b_r; q)_k} x^k.
 \end{align*}
  Keeping that $0<q<1$ in mind, noting that $0<q^{{n-k}\choose 2}\le 1$ and using the triangle inequality, we immediately find that
 \begin{align*}
 	&q^{n\choose 2}\left|{_{r+1}\phi_r}\left({{q^{-n},  a_1q^n, a_2,  \ldots, a_r }\atop{b_1, \ldots b_r }}; q, qx\right)\right|\\
 	&\le \sum_{k=0}^n 
 	\frac{|(q^{n-k+1}, a_1q^n, a_2, \ldots, a_r; q)_k|}{(q; q)_k|(q,  b_1,\ldots, b_r; q)_k|} |x|^k.
 \end{align*}
  Keeping that $0<q<1$ in mind,  using the triangle inequality and the first inequality in Proposition~\ref{liuequality} we deduce that
 \[
 |(a_1q^n; q)_k|\le (-|a_1|q^n; q)_k \le (-|a_1|; q)_k.
 \]
 Noting that $0<q^{n-k}\le 1$ for $k\in\{0, 1, \ldots, n\}$, using the triangle inequality and  the first inequality in Proposition~\ref{liuequality} we find that
 \[
 (q^{n-k+1}; q)_k\le (-q^{n-k+1}; q)_k\le (-q; q)_\infty.
 \] 

 Upon making the triangle inequality and noting that $(1+|a_j|q^l)\ge 1$, we see that, for $j\in\{2, 3,\ldots, r\}$,
 \[
 |(a_j; q)_k|\le\prod_{l=0}^{k-1}(1+|a_j|q^l)
 \le \prod_{l=0}^{\infty}(1+|a_j|q^l).
 \]
  Using the second inequality in Proposition~\ref{liuequality}, we find that, for $j\in\{1, 2,\ldots, r\}$, 
 \[
 |(b_j; q)_k|\ge (|b_j|; q)_\infty.
 \]
 It follows that
 \[
 \frac{|(q^{n-k+1}, a_1q^n, a_2, \ldots, a_r; q)_k|}{(q; q)_k|(q,  b_1,\ldots, b_r; q)_k|} |x|^k 
 \le \frac{(-q, -|a_2|, \ldots,-|a_r|; q)_\infty (-|a_1|; q)_k}{(|b_1|, \ldots, |b_r|;q)_\infty (q; q)_k}|\lambda|^k.
 \]
 Using this inequality the triangle inequality, we conclude that
 \begin{align*}
 &q^{n\choose 2}\left|{_{r+1}\phi_r}\left({{q^{-n},  a_1q^n, a_2,  \ldots, a_r }\atop{b_1, \ldots b_r }}; q, qx\right)\right|\\
 &\le \frac{(-q, -|a_2|, \ldots,-|a_r|; q)_\infty }{(|b_1|, \ldots, |b_r|;q)_\infty}
 \sum_{k=0}^\infty \frac{(-|a_1|; q)_k |\lambda|^k}{(q;q)_k}.
 \end{align*}
 Upon applying the $q$-binomial theorem to the right-hand side of the above inequality,
 we complete the proof of Proposition~\ref{convergenceseries}.
 \end{proof}

\begin{prop}\label{Liubounded} If $k$ is a nonnegative integer, then, for $\max\{|a|, |b|, |c|,|d|, |u|\}<1, A_k(\theta)$ is bounded on 
$[-\pi,\pi]$,where $A_k(\theta)$ is given by
\[
A_k(\theta)=\frac{h(\cos 2\theta; 1|q) h(\cos\theta; w|q)(te^{i\theta}, te^{-i\theta};q)_k}{h(\cos \theta; a, b,c, d,u|q)}.
\]	
\end{prop}
\section{the proof of Theorem~\ref{Liu12parametersint}}
The following general $q$-transformation formula is proved by us in \cite[Theorem~1.6]{Liu2013RamJ}.
\begin{prop}\label{Liu2013RJ} For an arbitrary sequence $\{A_n\}$ of complex numbers, under suitable convergence conditions, we have 
\begin{align*}
&\frac{(\alpha q, \alpha uv/q; q)_\infty}
{(\alpha u,\alpha v; q)_\infty}\sum_{n=0}^\infty A_n (q/u; q)_n (\alpha u)^n\\
&=\sum_{n=0}^\infty \frac{(1-\alpha q^{2n})(\alpha, q/u, q/v; q)_n (-\alpha uv/q)^n q^{n(n-1)/2}}{(1-\alpha)(q, \alpha u, \alpha v; q)_n}\\
&\qquad \qquad \qquad \qquad \qquad \qquad \times \sum_{k=0}^n 
\frac{(q^{-n}, \alpha q^n; q)_k(q^2/v)^k}{(q/v; q)_k}A_k.
\end{align*}	
\end{prop}
It is easily seen that  by choosing 
\[
A_k=\frac{(q/v, \beta,\delta; q)_k (vz/q)^k}
{(q, s, t, h; q)_k}
\]
in Proposition~\ref{Liu2013RJ}, we can derive the following proposition, which was used to prove may of the Hecke-type series identities 
in \cite{Liu2013IntJNT}.
\begin{prop}\label{Liu2013IJNT} For $\max\{|\alpha a|, |\alpha b|, |\alpha uvz/q|\}<1$, we have the $q$-formula 
	\begin{align*}
	&\frac{(\alpha q, \alpha uv/q; q)_\infty}
	{(\alpha u,\alpha v; q)_\infty}
	{_4\phi_3}\left({{q/u,q/v,\beta,\delta}\atop{s, t, h}}; q, \frac{\alpha uvz}{q}\right)\\
	&=\sum_{n=0}^\infty \frac{(1-\alpha q^{2n})(\alpha, q/u, q/v; q)_n (-\alpha uv/q)^n q^{n(n-1)/2}}{(1-\alpha)(q, \alpha u, \alpha v; q)_n}\\
	&\qquad \qquad \qquad \qquad  \times {_4\phi_3}\left({{q^{-n}, a^2bcdrq^{n-1},\beta,\delta}\atop{s, t, h}}; q, qz\right),
	\end{align*}
	provided the right-hand side of the above equation is an absolutely convergent series.	
\end{prop}
 Now we begin prove Theorem~\ref{Liu12parametersint} by using Propositions~\ref{liuequality}, \ref{convergenceseries}, \ref{Liubounded} and \ref{Liu2013IJNT},  as well as the Nassrallah--
Rahman integral in Theorem~\ref{RahInt}.
\begin{proof}
For notational simplicity, we will for the time being make use of $B_n$ to denote 
\begin{align*}
&\frac{(1-\alpha q^{2n})(\alpha, q/u, q/v; q)_n (-\alpha uv/q)^n q^{n(n-1)/2}}{(1-\alpha)(q, \alpha u, \alpha v; q)_n}\\
&\qquad \qquad \qquad  \times {_4\phi_3}\left({{q^{-n}, a^2bcdrq^{n-1},\beta,\delta}\atop{s, t, h}}; q, qz\right).
\end{align*}
Upon taking $\alpha=a^2bcdr/q, u=qe^{i\theta}/a$ and $v=qe^{-i\theta}/a$ in Proposition~\ref{Liu2013IJNT}, we find that
\begin{align}
&\sum_{n=0}^\infty \frac{(ae^{i\theta}, ae^{-i\theta};q)_n}{(abcdre^{i\theta}, abcdre^{-i\theta}; q)_n}B_n\label{liu:eqn1}\\
&=\frac{(a^2bcdr, bcdr; q)_\infty}{(abcdre^{i\theta}, abcdre^{-i\theta}; q)_\infty}
{_4\phi_3}\left({{ae^{i\theta}, ae^{-i\theta},\beta,\delta}\atop{s, t, h}}; q, bcdrz\right).\nonumber
\end{align}
Using  the triangular inequality and a simple calculation, we deduce
that
\begin{equation}
|(ae^{i\theta}, ae^{-i\theta}; q)_n| \le 
(-|a|; q)_\infty^2. \label{liu:eqn2}
\end{equation}
Keeping that $\max\{|a|,|b|,|c|,|d|,|r|\}<1$ in mind and using the triangular inequality, we have 
\begin{equation}
|(abcdre^{i\theta}, abcdre^{-i\theta}; q)_n|\ge 
(|abcdr|; q)^2_\infty.
 \label{liu:eqn3}
\end{equation}
It follows that
\[
\Big| \frac{(ae^{i\theta}, ae^{-i\theta};q)_n}{(abcdre^{i\theta}, abcdre^{-i\theta}; q)_n}\Big|
\le \frac{(-|a|; q)_\infty^2}{(|abcdr|; q)^2_\infty}.
\]
Assuming that $0<\lambda<1$ and $|z|\le \lambda$, then by Proposition~\ref{convergenceseries}, we find that
\[
q^{n(n-1)/2}\big|{_4\phi_3}\left({{q^{-n}, a^2bcdrq^{n-1},\beta,\delta}\atop{s, t, h}}; q, qz\right)\Big|
\]
is bounded by some constant independent of $n$ and $z$. Hence by the Hence by the ratio test we
know that the series on the left-hand side of 
(\ref{liu:eqn1}) is an absolutely and uniformly
convergent series of two variables $z$ and $\theta$ for $\max\{|a|,|b|,|c|,|d|,|r|\}<1$ and $|z|\le \lambda$ , where $0<\lambda<1.$

 Using Proposition~\ref{Liubounded}, with $t=0,$  we know that the following function of $\theta$ is bounded by some constant independent of $\theta:$
\[
\frac{h(\cos 2\theta; 1|q)h(\cos \theta; abcdr|q)}{h(\cos \theta; a, b, c, d, r|q)}.
\]
Upon multiplying both sides of (\ref{liu:eqn1}) by the above function of $\theta,$  using the identity $(x; q)_n (xq^n; q)_\infty=(x; q)_\infty$ in  \cite[Eq.(1.2.30)]{GasperRahman2004}, we obtain
\begin{align*}
&\frac{h(\cos 2\theta; 1|q)}{h(\cos \theta; a, b, c, d, r|q)}
{_4\phi_3}\left({{ae^{i\theta}, ae^{-i\theta},\beta,\delta}\atop{s, t, h}}; q, bcdrz\right)\\
&=\frac{1}{(a^2bcdr, bcdr; q)_\infty}\sum_{n=0}^\infty\frac{h(\cos 2\theta; 1|q)h(\cos \theta; abcdrq^n|q)}{h(\cos \theta; a, b, c, d, r|q)} B_n.
\end{align*}
Through the above discussion, we know that the right-hand side of the above equation
is convergent uniformly and absolutely on $[0, \pi]$. Hence we can integrate the above
equation term by term over $[0,\pi]$ to obtain
\begin{align*}
&\int_{0}^{\pi}\frac{h(\cos 2\theta; 1|q)}{h(\cos \theta; a, b, c, d, r|q)}
{_4\phi_3}\left({{ae^{i\theta}, ae^{-i\theta},\beta,\delta}\atop{s, t, h}}; q, bcdrz\right)d\theta\\
&=\frac{1}{(a^2bcdr, bcdr; q)_\infty}\sum_{n=0}^\infty B_n \int_{0}^\pi \frac{h(\cos 2\theta; 1|q)h(\cos \theta; abcdrq^n|q)}{h(\cos \theta; a, b, c, d, r|q)} d\theta.
\end{align*}
Using the Nassrallah--Rahman $q$-beta integral in Theorem~\ref{RahInt}, we immediately conclude that
\begin{align*}
&\int_{0}^\pi \frac{h(\cos 2\theta; 1|q)h(\cos \theta; abcdrq^n|q)}{h(\cos \theta; a, b, c, d, r|q)} d\theta\\
&=\frac{2\pi (abcd, abcr, abdr, acdr, bcdr; q)_\infty(ab, ac, ad, ar; q)_n}{(q, ab, ac, ad, ar, bc, bd, br, cd, cr, dr; q)_\infty
(abcd, abcr, abdr, acdr; q)_n}.
\end{align*}
Combining the two equations above we complete the proof of Theorem~\ref{Liu12parametersint}. From the process of proof of Theorem~\ref{Liu12parametersint} , we find that this theorem also holds when $|z|=1$,
provided the series on the right-hand side of the above equation converges absolutely
for $|z|=1$.
\end{proof}
\section{the proof of Theorem~\ref{Liu12parametersdoub}}
\begin{proof}
 It follows that from (\ref{liu:eqn2}), by the ratio test, we find that the following series is absolutely and uniformly
convergent on $[0,\pi]$ when $|bcdrz|<1:$
\[
{_4\phi_3}\left({{ae^{i\theta}, ae^{-i\theta},\beta,\delta}\atop{s, t, h}}; q, bcdrz\right).
\]
 Using Proposition~\ref{Liubounded}, with $t=w=0,$ we know that the following function of $\theta$ is bounded by some constant independent of $\theta:$
\[
\frac{h(\cos 2\theta; 1|q)}{h(\cos \theta; a, b, c, d, r|q)}.
\]
Hence we can evaluate the  integral on the left-hand side of the equation in Theorem~\ref{Liu12parametersint} term by term to obtain
\begin{align*}
&\int_{0}^{\pi}\frac{h(\cos 2\theta; 1|q)}{h(\cos \theta; a, b, c, d, r|q)}
{_4\phi_3}\left({{ae^{i\theta}, ae^{-i\theta},\beta,\delta}\atop{s, t, h}}; q, bcdrz\right)d\theta\\
&=\sum_{n=0}^\infty \frac{(\beta, \delta; q)_n (bcdrz)^n}{(q, s, t, h; q)_n}
\int_{0}^{\pi}\frac{h(\cos 2\theta; 1|q)}{h(\cos \theta; aq^n, b, c, d, r|q)}
d\theta.
\end{align*}
Using the Ismail--Stanton--Viennot integral in Theorem~\ref{IsmailStanton}, we immediately deduce that
\begin{align*}
&\int_{0}^{\pi}\frac{h(\cos 2\theta; 1|q)}{h(\cos \theta; aq^n, b, c, d, r|q)}
d\theta\\
&=\frac{2\pi(abcdq^n, abcrq^n; q)_\infty}
{(q, abq^n, acq^n, adq^n, arq^n, bc, bd,br, cd, cr; q)_\infty}
{_3\phi_2}\left({{abq^n, acq^n, bc}\atop{abcdq^n, abcrq^n}}; q, dr\right).
\end{align*}
It follows that
\begin{align*}
&\int_{0}^{\pi}\frac{h(\cos 2\theta; 1|q)}{h(\cos \theta; a, b, c, d, r|q)}
{_4\phi_3}\left({{ae^{i\theta}, ae^{-i\theta},\beta,\delta}\atop{s, t, h}}; q, bcdrz\right)d\theta\\
&=\frac{2\pi(abcd, abcr; q)_\infty}
{(q, ab, ac, ad, ar, bc, bd,br, cd, cr; q)_\infty}\\
&\quad \times
\sum_{n=0}^\infty \frac{(ab, ac, ad, ar, \beta, \delta; q)_n (bcdrz)^n}{(q, abcd, abcr, s, t, h; q)_n}{_3\phi_2}\left({{abq^n, acq^n, bc}\atop{abcdq^n, abcrq^n}}; q, dr\right).
\end{align*}
Equating this with the right-hand side of the equation in Theorem~\ref{Liu12parametersint}, we complete the proof of Theorem~\ref{Liu12parametersdoub}.
\end{proof}
\section{some special cases of Theorem~\ref{Liu12parametersint}}
\begin{thm}\label{SalamIsmail} For $\max\{|a|,|b|,|c|, |d|, |r|\}<1,$ we have
\begin{align*}
&\int_{0}^{\pi}\frac{h(\cos 2\theta; 1|q)h(\cos \theta; abcdrz|q)}{h(\cos \theta; a, b, c, d, r|q)}d\theta\\
&=\frac{2\pi (abcd, abcr, abdr, acdr, a^2bcdrz, bcdrz; q)_\infty}{(q, ab, ac, ad, ar, bc, bd, br, cd, cr, dr, a^2bcdr; q)_\infty}\\
&\qquad \times {_8W_7}\left(a^2bcdr/q; ab, ac, ad, ar, 1/z; bcdrz\right).
\end{align*}
\end{thm}
 This integral is slightly different in form from the Nassrallah--Rahman $q$-beta integral in Theorem~\ref{NassRahInt}.
Putting $z=1$, noting that $(1; q)_k = \delta_{0, k}$, the above equation
immediately reduces to the Nassrallah--Rahman $q$-beta integral in Theorem~\ref{RahInt}. 

Now 
we prove Theorem~\ref{SalamIsmail} by using Theorem~\ref{Liu12parametersint}, the $q$-Gauss summation formula and $q$-Chu--Vandermonde summation formula.
\begin{proof}
When $s=a^2bcdrz$ and $(\beta,\delta)=(t, h)$,
the $_4\phi_3$ series under the integral sign in Theorem~\ref{Liu12parametersint} reduces to 
\[
{_2\phi_1}\left({{ae^{i\theta}, ae^{-i\theta}}\atop{a^2bcdrz}}; q, bcdrz \right)
=\frac{(abcdrze^{i\theta}, abcdrz e^{-i\theta}; q)_\infty}{(a^2bcdrz, bcdrz; q)_\infty}
\]
by the $q$-Gauss summation \cite[Eq.(1.5.1)]{GasperRahman2004},  and the $_4\phi_3$ series on the right-hand side of the equation in Theorem~\ref{Liu12parametersint} becomes 
\[
{_2\phi_1}\left({{q^{-n}, a^2bcdrq^{n-1}}\atop{a^2bcdrz}}; q, qz \right)
=\frac{(1/z; q)_n (-z)^n}{(a^2bcdrz; q)_n} 
q^{-n(n-1)/2}
\]
by the $q$-Chu--Vandermonde summation \cite[Eq.(1.5.2)]{GasperRahman2004}. This completes the proof of Theorem~\ref{SalamIsmail}.
\end{proof}
The corresponding series form of Theorem~\ref{SalamIsmail} is stated in the following theorem,
which is equivalent to the formula due to Al--Salam and Ismail \cite[Eq.(4.4)]{Al--SalamIsmail1988}.
\begin{thm}\label{SalamIsmailnew} 
 For $\max\{|a|,|b|,|c|, |d|, |r|\}<1,$ we have
 \begin{align*}
 &\sum_{n=0}^\infty \frac{(ab, ac, ad, ar; q)_n (bcdrz)^n}{(q, a^2bcdrz, abcd, abcr; q)_n}
 {_3\phi_2}\left({{abq^n, acq^n, bc}\atop{abcdq^n, abcrq^n}}; q, dr\right)\\
 &=\frac{(abdr, acdr; q)_\infty}
 {(dr, a^2bcdr; q)_\infty}
 {_8W_7}\left(a^2bcdr/q; ab, ac, ad, ar, 1/z; bcdrz\right).
 \end{align*}
\end{thm}

\begin{thm}\label{SLiu12parametersinta}
For $\alpha=a^2bcdr/q$ and 	$\max\{|a|,|b|,|c|, |d|, |r|\}<1,$  we have 
\begin{align*}
&\int_{0}^{\pi} \frac{h(\cos 2\theta; 1|q)}
{h(\cos \theta; a, b, c, d, r|q)}
{_3\phi_2}\left({{ae^{i\theta}, ae^{-i\theta}, \alpha uv/q}\atop{\alpha u, \alpha v}}; q, bcdr\right) d\theta\\
&=\frac{2\pi (abcd, abcr, abdr, acdr; q)_\infty}{(q, ab, ac, ad, ar, bc, bd, br, cd, cr, dr, a^2bcdr; q)_\infty}\\
&\quad \times 
\sum_{n=0}^\infty \frac{(1-\alpha q^{2n})(\alpha, ab, ac, ad, ar, q/u, q/v; q)_n}{(1-\alpha)(q, abcd, abcr, abdr, acdr, \alpha u, \alpha v; q)_n}
\left(-\frac{\alpha bcdruv}{q}\right)^n q^{n(n-1)/2}.
\end{align*}
\end{thm}
\begin{proof}
Noting that $\alpha=a^2bcdr/q$, choosing $s=\alpha u, t=\alpha v, \beta=\alpha uv/q$ 
and $\delta=h, z=1$ in Theorem~\ref{Liu12parametersint}, we deduce that
\begin{align*}
&\int_{0}^{\pi} \frac{h(\cos 2\theta; 1|q)}
{h(\cos \theta; a, b, c, d, r|q)}
{_3\phi_2}\left({{ae^{i\theta}, ae^{-i\theta}, \alpha uv/q}\atop{\alpha u, \alpha v}}; q, bcdr\right) d\theta\\
&=\frac{2\pi (abcd, abcr, abdr, acdr; q)_\infty}{(q, ab, ac, ad, ar, bc, bd, br, cd, cr, dr, a^2bcdr; q)_\infty}\\
&\quad \times 
\sum_{n=0}^\infty \frac{(1-\alpha q^{2n})(\alpha, ab, ac, ad, ar; q)_n}{(1-\alpha)(q, abcd, abcr, abdr, acdr; q)_n}(-bcdr)^n q^{n(n-1)/2}\\
&\qquad \qquad \qquad \qquad \qquad \qquad \qquad \times 
{_3\phi_2} \left({{q^{-n}, \alpha q^n, \alpha uv/q}\atop{\alpha u, \alpha v}}; q, q\right).
\end{align*}
By the well-known $q$-Pfaff--Saalsch\"{u}tz formula \cite[Eq.(1.7.2)]{GasperRahman2004}, we find that
\[
{_3\phi_2} \left({{q^{-n}, \alpha q^n, \alpha uv/q}\atop{\alpha u, \alpha v}}; q, q\right)
=\frac{(q/u, q/v; q)_n}{(\alpha u, \alpha v; q)_n}\left(\frac{\alpha uv}{q}\right)^n.
\]
Upon combining the two equations above we complete the proof of Theorem~\ref{SLiu12parametersinta}.
\end{proof}
The corresponding series form of Theorem~\ref{SLiu12parametersinta} is stated in the following theorem.
\begin{thm}\label{SLiu12parametersintb}
	For $\alpha=a^2bcdr/q$ and 	$\max\{|a|,|b|,|c|, |d|, |r|\}<1,$  we have 
	\begin{align*}
	&\sum_{n=0}^\infty \frac{(\alpha uv/q, ab, ac, ad, ar; q)_n (bcdr)^n}{(q, \alpha u, \alpha v, abcd, abcr; q)_n}
	{_3\phi_2}\left({{abq^n, acq^n, bc}\atop{abcdq^n, abcrq^n}}; q, dr\right)\\
	&=\frac{(abdr, acdr; q)_\infty}{( dr, q\alpha; q)_\infty}\\
	&\quad \times 
	\sum_{n=0}^\infty \frac{(1-\alpha q^{2n})(\alpha, ab, ac, ad, ar, q/u, q/v; q)_n}{(1-\alpha)(q, abcd, abcr, abdr, acdr, \alpha u, \alpha v; q)_n}
	\left(-\frac{\alpha bcdruv}{q}\right)^n q^{n(n-1)/2}.
	\end{align*}
\end{thm}

\begin{thm}\label{SLiu12parametersintc}
	For $\alpha=a^2bcdr/q$ and 	$\max\{|a|,|b|,|c|, |d|, |r|\}<1,$  we have 
	\begin{align*}
	&\int_{0}^{\pi} \frac{h(\cos 2\theta; 1|q)}
	{h(\cos \theta; a, b, c, d, r|q)}
	{_3\phi_2}\left({{ae^{i\theta}, ae^{-i\theta}, 0}\atop{\sqrt{q\alpha}, -\sqrt{q\alpha}}}; q, bcdr\right) d\theta\\
	&=\frac{2\pi (abcd, abcr, abdr, acdr; q)_\infty}{(q, ab, ac, ad, ar, bc, bd, br, cd, cr, dr, a^2bcdr; q)_\infty}\\
	&\quad \times 
	\sum_{n=0}^\infty \frac{(1-\alpha q^{4n})(\alpha, ab, ac, ad, ar; q)_{2n} (q; q^2)_n}{(1-\alpha)(q, abcd, abcr, abdr, acdr; q)_{2n} (q\alpha; q^2)_n}
	(-\alpha)^n (bcdr)^{2n} q^{3n^2-n}.
	\end{align*}
\end{thm}
\begin{proof}
It is easily seen that the identity of Verma and Jain \cite[Eq.(2.28)]{VermaJain1980} is equivalent to the summation 
\begin{equation*}
{_3\phi_2} \left({{q^{-n}, \alpha q^n, 0} \atop {\sqrt{q\alpha}, -\sqrt{q\alpha}}} ;  q, q \right)
=\begin{cases} 0 &\text {if $n$ is odd}\\
(-1)^l q^{l^2} \frac{(q; q^2)_l \alpha^l }{(q\alpha; q^2)_l}, & \text{if $n=2l$}
\end{cases}
\end{equation*}
Using this summation to Theorem~\ref{Liu12parametersint}, we complete the proof of Theorem~\ref{SLiu12parametersintc}.
\end{proof}
The corresponding series form of Theorem~\ref{SLiu12parametersintc} is stated in the following theorem.

\begin{thm}\label{SLiu12parametersintd}
	For $\alpha=a^2bcdr/q$ and 	$\max\{|a|,|b|,|c|, |d|, |r|\}<1,$  we have 
	\begin{align*}
	&\sum_{n=0}^\infty \frac{( ab, ac, ad, ar; q)_n (bcdr)^n}{(q, \sqrt{q\alpha}, -\sqrt{q\alpha}, abcd, abcr; q)_n}
	{_3\phi_2}\left({{abq^n, acq^n, bc}\atop{abcdq^n, abcrq^n}}; q, dr\right)\\
	&=\frac{(abdr, acdr; q)_\infty}{( dr, q\alpha; q)_\infty}\\
	&\quad \times 
	\sum_{n=0}^\infty \frac{(-\alpha)^n(1-\alpha q^{4n})(\alpha, ab, ac, ad, ar; q)_{2n} (q; q^2)_n}{(1-\alpha)(q, abcd, abcr, abdr, acdr; q)_{2n} (q\alpha; q^2)_n}
	(bcdr)^{2n} q^{3n^2-n}.
	\end{align*}
	\end{thm}

\begin{thm}\label{SLiu12andrews}
	For $\alpha=a^2bcdr/q$ and 	$\max\{|a|,|b|,|c|, |d|, |r|\}<1,$  we have 
	\begin{align*}
	&\int_{0}^{\pi} \frac{h(\cos 2\theta; 1|q)}
	{h(\cos \theta; a, b, c, d, r|q)}
	{_4\phi_3}\left({{ae^{i\theta}, ae^{-i\theta}, \sqrt{\lambda},-\sqrt{\lambda}}\atop{\sqrt{q\alpha}, -\sqrt{q\alpha}, \lambda}}; q, bcdr\right) d\theta\\
	&=\frac{2\pi (abcd, abcr, abdr, acdr; q)_\infty}{(q, ab, ac, ad, ar, bc, bd, br, cd, cr, dr, a^2bcdr; q)_\infty}\\
	&\quad \times 
	\sum_{n=0}^\infty \frac{(1-\alpha q^{4n})(\alpha, ab, ac, ad, ar; q)_{2n} (q, q\alpha/\lambda; q^2)_n}{(1-\alpha)(q, abcd, abcr, abdr, acdr; q)_{2n} (q\alpha, q \lambda; q^2)_n}
	 (bcdr)^{2n} \lambda^n q^{2n^2-2n}.
	\end{align*}
\end{thm}
\begin{proof}Recall that the $q$-Watson formula due to Andrews \cite{Andrews1976} states that
\begin{equation*}
{_4\phi_3} \left({{q^{-n}, \alpha q^n, \sqrt{\lambda}, -\sqrt{\lambda}} \atop {\sqrt{q \alpha}, -\sqrt{q \alpha}, \lambda}} ;  q, q \right)
=\begin{cases} 0 &\text {if $n$ is odd}\\
\frac{(q, \alpha q/\lambda; q^2)_{n/2} \lambda^{n/2}}{(q\alpha, q\lambda; q^2)_{n/2}}, & \text{if $n$ is even}.
\end{cases}
\end{equation*}	
For an alternative proof of this identity, please see \cite[Proposition~8.3]{Liu2014JMAA}. Using this
equation to Theorem~\ref{Liu12parametersint}, we complete the proof of Theorem~\ref{SLiu12andrews}.	
\end{proof}	
The corresponding series form of Theorem~\ref{SLiu12andrews} is stated in the following theorem.
\begin{thm}\label{SLiu12andrewsa}
	For $\alpha=a^2bcdr/q$ and 	$\max\{|a|,|b|,|c|, |d|, |r|\}<1,$  we have 
	\begin{align*}
	&\sum_{n=0}^\infty \frac{( \sqrt{\lambda}, -\sqrt{\lambda}, ab, ac, ad, ar; q)_n (bcdr)^n}{(q, \sqrt{q\alpha}, -\sqrt{q\alpha}, \lambda, abcd, abcr; q)_n}
	{_3\phi_2}\left({{abq^n, acq^n, bc}\atop{abcdq^n, abcrq^n}}; q, dr\right)\\
	&=\frac{(abdr, acdr; q)_\infty}{( dr, q\alpha; q)_\infty}\\
	&\quad \times 
		\sum_{n=0}^\infty \frac{(1-\alpha q^{4n})(\alpha, ab, ac, ad, ar; q)_{2n} (q, q\alpha/\lambda; q^2)_n}{(1-\alpha)(q, abcd, abcr, abdr, acdr; q)_{2n} (q\alpha, q \lambda; q^2)_n}
	(bcdr)^{2n} \lambda^n q^{2n^2-2n}.
	\end{align*}
\end{thm}

\section{Acknowledgments}
I am grateful to the anonymous referee and the editor  for careful reading of the manuscript and many invaluable suggestions and comments.

\end{document}